\documentclass{amsart}

\usepackage{tikz} \usetikzlibrary{graphs}     \usetikzlibrary{decorations.pathreplacing}
\usetikzlibrary {graphs.standard, quotes}
\usetikzlibrary{external}
\tikzset{external/optimize=false}

\usepackage[T1]{fontenc}
\usepackage{comment}
\usepackage{appendix}
\usepackage{graphicx}
\usepackage{pdfpages}
\usepackage{amssymb}
\usepackage{hyperref}

\newtheorem{theorem}{Theorem}[section]
\newtheorem{corollary}[theorem]{Corollary}
\newtheorem{lemma}[theorem]{Lemma}
\newtheorem{proposition}[theorem]{Proposition}

\theoremstyle{definition}
\newtheorem{definition}[theorem]{Definition}
\newtheorem{example}[theorem]{Example}
\theoremstyle{remark}
\newtheorem{remark}[theorem]{Remark}
\newtheorem*{claim}{Claim}
\usepackage{subcaption}
\usepackage{float}
\usepackage{epic}
\usepackage{import}
\usepackage{pdfpages}
\usepackage{epsfig}
\usepackage{tikz-cd}
\usepackage{adjustbox}

\usepackage{float}
\usepackage{lipsum}

\title{Minimal rational graphs admitting a \QHD smoothing}
\author{M\'arton Beke}
\address{Alfréd Rényi Institute of Mathematics, Budapest\\ University of Technology and Economics, Budapest, Hungary}
\email{bekem@renyi.hu}

\def\Z{\mathbb{Z}}

\def\Q{\mathbb{Q}}
\def\QHD{{$\mathbb{Q}\text{HD}$ }}
\def\ABC{\mathcal A,\mathcal B,\mathcal C}
\def\ABCt{\mathcal A^3,\mathcal B^3,\mathcal C^3}
\def\ABCf{\mathcal A^4,\mathcal B^4,\mathcal C^4}
\DeclareMathOperator{\rk}{{rk}}

\begin{document}

\begin{abstract}
    Using the picture deformation technique of De Jong-Van Straten we show that no singularity whose resolution graph has 3 or 4 large nodes, i.e., nodes satisfying $d(v)+e(v)\leq -2$, has a \QHD smoothing.
    This is achieved by providing a general reduction algorithm for graphs with \QHD smoothings, and enumeration.
    New examples and families are presented, which admit a combinatorial \QHD smoothing, i.e. the incidence relations for a sandwich presentation can be satisfied.
    We also give a new proof of the Bhupal-Stipsicz theorem on the classification of  weighted homogeneous singularities admitting \QHD smoothings with this method by using cusp singularities.
\end{abstract}
\maketitle

\section{Introduction}
The study of complex surface singularities is a central area of interest in algebraic geometry. A key problem is understanding the various ways a singularity can be deformed into a smooth surface through a process called smoothing. This paper investigates a specific type of deformation known as a rational homology disk (\QHD) smoothing. These are the topologically simplest kinds of deformations, giving us hope of a complete classification. The conjecture is attributed to Wahl, who had an unpublished list of singularities admitting a \QHD smoothing, stating that this list is complete.
This paper presents further evidence for  Wahl's Conjecture by showing the following:

\begin{theorem}\label{thm:fo}
	No complex surface singularity whose resolution graph has at most three large nodes of any degree, or at most four nodes of degree 3, and no further nodes can admit a \QHD smoothing.
\end{theorem}
\noindent Moreover, we give a new proof of the classification theorem for weighted homogeneous graphs.
\begin{theorem}[Bhupal-Stipsicz 2010 \cite{bhupal2011weighted}]
	The star-shaped graphs admitting a \QHD smoothing are precisely the graphs depicted on Figure~\ref{fig:WMN} and Figures~\ref{fig:C6}-\ref{fig:4-star}.
\end{theorem}

Besides the algebro-geometric perspective, these spaces are also important in low-dimensional topology motivated by the rational blowdown operation of Fintushel-Stern \cite{fintushel1996knots}, which uses \QHD fillings of linear graphs and a cut-and-paste operation to generate topologically smaller exotic 4-manifolds. In a way this problem also tries to classify all possible rational blowdowns of 4-manifolds.

Our result can also be interpreted as a generalization of an observation by Bhupal--Stipsicz (\cite[Theorem 2.5, Remark 2.6]{bhupal2013smoothings}): a star-shaped graph, which admits a \QHD smoothing cannot have a too negative node, which is the same as the  large  nodes discussed here ($e(v)+d(v)\leq -2$).

To achieve these results, we use the deformation theory of sandwiched singularities developed by de Jong and van Straten in \cite{de1998deformation}. This reduces the question of finding deformations of surface singularities to deformations of curves with some intersections prescribed by the surface singularity. We will be mostly dealing with the combinatorial properties of these intersection patterns.
Using Donaldson's diagonalizability theorem strong necessary conditions were imposed on graphs admitting a \QHD smoothing by Stipsicz-Szabó-Wahl \cite{stipsicz2008rational}, in particular the study of graphs with \QHD smoothings is essentially reduced to the  families $\mathcal G,\mathcal W,\mathcal N,\mathcal M,\mathcal A,\mathcal B,\mathcal C$, a fact we will rely on in what follows.
The proof of Theorem~\ref{thm:fo} is as follows:
under the assumption of a graph having large nodes and a \QHD smoothing, we find a reduction algorithm (\ref{subsec:algo}), so if a graph admits (the combinatorics of a) \QHD smoothing, then some smaller graph does so too. This, and information about the general shape  of our graphs reduces the problem to a finite number of graphs (Lemma~\ref{lem:finite}). We then show (Theorem~\ref{thm:zksq}) that any graph that admits a combinatorial solution to its \QHD smoothing problem satisfies $Z_K^2+n=0$, where $n$ is the number of vertices of the graph, and $Z_K$ is the anticanonical cycle. Using this statement, we can search through all possible graphs to finish the proof of the theorem.

The proof of the Bhupal-Stipsicz theorem is a case analysis of possible solutions to the \QHD smoothing problem of graphs in $\ABC$ (since the other cases have been settled already, see \cite[Section 8]{stipsicz2008rational}). This argument is notably simpler and shorter than the original proof, which utilizes the combinatorics of symplectic curves and McDuff's theorem.
Accordingly the result is weaker, since we only rule out the $\ABC$ graphs with one star from having a \QHD smoothing, whereas \cite{bhupal2011weighted} rules them out from having a symplectic filling.
This gap is filled by new results of Plamenevskaya-Starkston.
In \cite{plamenevskaya2023unexpected} they expand the de Jong-van Straten theory into the symplectic case, if the corresponding curve singularity has only smooth branches, and in \cite{plamenevskaya2025sandwiched} to the general case. This means, that if the intersection pattern for a \QHD smoothing is impossible then the graph cannot admit a symplectic \QHD filling.

The paper is organized as follows: In Section 2, we review details on the $\ABC$ families, and the deformation theory for sandwiched singularities.
Section 3 gives specific introduction to minimal rational graphs, and describes and proves the reduction algorithm.
Section 4 shows some necessary conditions for a graph to admit a combinatorial \QHD smoothing.
Section 5 is dedicated to examples of graphs, which do admit combinatorial \QHD smoothings, and Section 6 consists of an introduction to the case when the curves can have cusp singularities, and the proof of the Bhupal-Stipsicz theorem.

\subsection{Notation}
A number $k$ on the edge of a graph $\Gamma$ means a path consisting of $k$ vertices framed $-2$. In some cases, $k$ is allowed to be $-1$; in this case, we identify the two endpoints of the edge with $k$ and frame it $e+f+2$, where $e,f$ are the framings of the two endpoints of the edge with $k$ over it.
\begin{figure}[h!]
	\centering
	\includegraphics{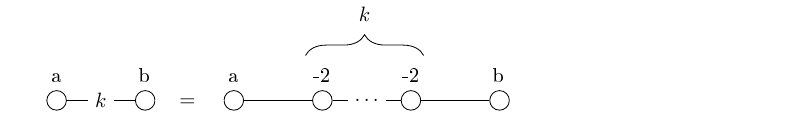}\hspace{-2cm}
	\includegraphics{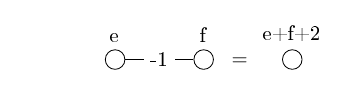}
		%
		\label{fig:notation}
	\end{figure}
	\subsection{Acknowledgments}
	The author is thankful to András Némethi, and András Stipsicz for their continued help and support with this project. The author is thankful for enlightening conversations to Marco Golla, László Koltai and Olga Plamenevskaya.
	
	The project supported by the Doctoral Excellence Fellowship Programme (DCEP)  is funded by the National Research Development and Innovation Fund of the Ministry of Culture and Innovation and the Budapest University of Technology and Economics and by ERC Advanced Grant KnotSurf4d.

	\section{Preliminaries}
	\subsection{Graphs Admitting a \QHD smoothing}
	In \cite{stipsicz2008rational} Stipsicz--Szabó--Wahl derive strong restrictions on the resolution graph of a singularity if it admits a \QHD smoothing.
	\begin{definition}
		A resolution graph $\Gamma$ on $n$ vertices is called a \textit{symplectic plumbing tree} if its intersection lattice $Q_\Gamma$ admits an embedding into the negative definite trivial lattice $(\mathbb Z\langle e_1,\dots,e_n\rangle,n\langle-1\rangle)$ such that $K=\sum e_i$ represents the anticanonical cycle of $Q_\Gamma$.
	\end{definition}
	\begin{theorem}[{\cite[Corollary 2.5]{stipsicz2008rational}}]\label{thm:necessary}
		If a complex surface singularity admits a \QHD smoothing, its resolution graph is a symplectic plumbing tree.
	\end{theorem}
	The minimal resolution graphs with this property are completely classified; they belong to one of 7 classes.
	The first, $\mathcal{G}$, consists of linear graphs, with framings given by the negatives of the Hirzebruch-Jung continued fraction coefficients of $\frac{p^2}{pq-1}$ for $p>q>0$ coprime integers. The next three $\mathcal{W, N, M}$ are depicted in Figure~\ref{fig:WMN}.
	
	\begin{figure}[h!]

		\vspace{-1cm}
		\includegraphics{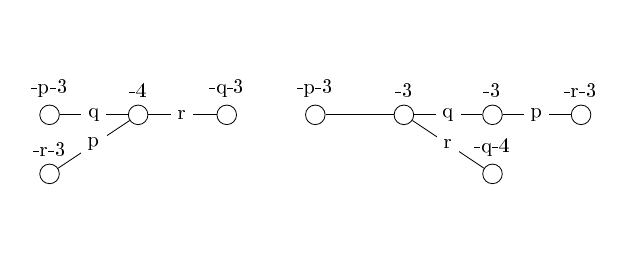}
			\includegraphics{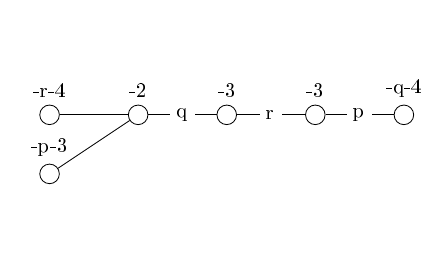}
			%
				\caption{The graphs of $\mathcal W,\mathcal N$ and $\mathcal M$ where $p,q,r\geq -1$.}
				\label{fig:WMN}
			\end{figure}
			
			The remaining $\mathcal{A, B, C}$ are defined by repeated blowups of an edge next to the unique $-1$ vertex or the $-1$ vertex itself, beginning from the graph of Figure~\ref{fig:ABC}. After the blowups, the framing of the $-1$ vertex (which is always unique) is changed to $-4,-3,-2$ to obtain an element of the $\mathcal{A, B, C}$ classes, respectively.
			\begin{figure}[h!]
				\includegraphics{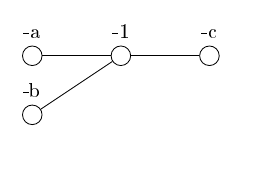}
					\vspace{-1cm}
					\caption{The starting graph for the construction of the $\ABC$ families. $(a,b,c)=(3,3,3),(2,4,4),(2,3,6)$ respectively.}
					\label{fig:ABC}
				\end{figure}
				\begin{theorem}[{\cite[Theorem 1.8]{stipsicz2008rational}}]
					The set of minimal symplectic plumbing trees $\mathcal S$ equals $\mathcal G\cup\mathcal W\cup\mathcal N\cup\mathcal M\cup\mathcal A\cup\mathcal B\cup\mathcal C$.
				\end{theorem}
				This is only a necessary condition; for the first four classes, a converse of Theorem~\ref{thm:necessary} also holds.
				\begin{theorem}[{\cite[Theorem 1.10]{stipsicz2008rational}}]\label{thm:existence}
					For every element  $\Gamma\in\mathcal G\cup\mathcal W\cup\mathcal N\cup\mathcal M$, there is a complex surface singularity with resolution graph $\Gamma$ admitting a \QHD smoothing.
				\end{theorem}
				The case of the $\ABC$ classes is only partially resolved, but there is a complete classification for weighted homogeneous singularities in \cite[Theorems 1.4, 1.6]{bhupal2011weighted}, which will be proven in Section~\ref{sect:bpst}.
				
				\subsection{Deformations of sandwiched singularities}
				\begin{definition}
					A resolution graph $\Gamma$ is called a \textit{sandwich graph} if it arises as a subgraph of a graph that can be blown down to the empty graph.
					A singularity is called \textit{sandwiched} if its resolution graph is a sandwich graph.
				\end{definition}
				The deformation theory of sandwiched singularities 
				essentially reduces to deformations of plane curves, as will be explained momentarily.
				We rely on the picture deformation method developed in \cite{de1998deformation}.
				
				Let $\rho:(\tilde A,E)\to(\mathbb C^2,0)$ be obtained by point blowups starting from $(\mathbb C^2,0)$ and the singularity $X$ be obtained by contracting the non-$(-1)$ curves from $\tilde A$.
				Choose one transverse \textit{curvetta} $\tilde C_i$ on each $-1$ curve of $\tilde A$.
				Now $\tilde A$ can be seen as a good embedded resolution of $\rho(\tilde C)$.
				
				Reversing the blowup process, De Jong and Van Straten define:
				\begin{definition}[{\cite[Definitions 1.3, 4.1]{de1998deformation}}]
					A \textit{decorated curve} is a pair $(C,l)$, where $C=\cup_I C_i\subset(\mathbb C^2,0)$ is a curve with $|I|$ irreducible components and a function $l:I\to\mathbb N$ assigning to each irreducible component a number. 
					Let $m(i)$ denote the sum of the multiplicities of the $i$th branch in the minimal resolution of $C$. Then 
					\begin{equation}\label{eq:multiplicity}
						m(i)\leq l(i)
					\end{equation} for each $i\in I$.
					For discussing deformations, it is useful to encode the decoration $l$ as a divisor on the normalization $\tilde C$ of $C$.
				\end{definition}
				From a decorated curve, we obtain a surface singularity as follows:
				\begin{definition}[{\cite[Definition 1.4]{de1998deformation}}]
					Let $(C,l)$ be a decorated curve.
					The modification $\tilde A(C,l)\to(\mathbb C^2,0)$ determined by $(C,l)$ is obtained from the minimal resolution of $C$ by $l(i)-m(i)$ further blowups at the $i$th branch of $C$.
					The analytic set $X(C,l)$ is obtained by contracting the exceptional curves not intersecting any branch of $C$ in $\tilde A(C,l)$.
				\end{definition}
				\begin{remark}
					If during the blowup process, we don't reach the minimal \textit{good} resolution of $C$, the analytic set can have multiple singular points.
					
					Note that the same surface singularity may be presented in multiple ways from different decorated curves.
				\end{remark}
				Next, they define the corresponding notion of deformation of sandwich singularities from the decorated curve data.
				\begin{definition}[{\cite[Definition 4.2]{de1998deformation}}]
					A one-parameter deformation of a decorated curve $(C,l)$ is a $\delta$-constant deformation $C_\Delta\to \Delta$ of $C$ (where $\Delta:=(\mathbb C,0)$) and a flat deformation $l_\Delta\subset \tilde C\times \Delta$ (where $\tilde C$ is the normalization of $C$) under the multiplicity condition of Equation~\ref{eq:multiplicity}.
					
					The multiplicity condition means a divisor on each fiber $l_t=l_\Delta\cap C_t=\sum_{p\in C_t}m_pp$ subject to $\sum_{p\in (C_t)_i}m_p=\sum_{p\in (C_0)_i}m_p=l(i)$, where $(C)_i$ denotes the $i$th irreducible component of the curve $C$.
				\end{definition}
				The important theorem for us is the following:
				\begin{theorem}[{\cite[Theorem 4.4]{de1998deformation}}]\label{thm:picdef}
					For every one-parameter deformation of $(C,l)$, there exists a one-parameter deformation of $X(C,l)$, and every one-parameter deformation of $X(C,l)$ appears in this form!
				\end{theorem}
				Mostly we are interested in smoothing components. It is folklore that for rational singularities (which sandwich singularities are, see \cite{spivakovsky1990sandwiched}) every component in the base space of the semi-universal deformation space of $X$ (denoted $\mathcal B(X)$) is a smoothing component. That is, smoothing occurs over a generic point of this component.
				A one-parameter deformation is always induced from a well-defined component of $\mathcal B(X)$, and the generic fiber can be described by a particularly simple special case of the previous theorem.
				\begin{definition}[{\cite[Definition 4.6]{de1998deformation}}]
					A one parameter deformation $(C_\Delta,l_\Delta)$ is called a \textit{picture deformation}, if for generic $t\neq 0$ the divisor $l_t$ is reduced.
				\end{definition}
				\begin{remark}
					It follows that the generic fiber of the $\delta$-constant deformation $C_t$ only has transverse $m$-tuple points as singularities.
					Following \cite{de1998deformation} we call $1$-tuple points "\textit{free}".
					
					The general fiber of a picture deformation in general consists of a collection of curves, with the $m$-tuple points marked for $m>1$, and possibly some further marked points on each component, so that in total, each component $C_i$ contains $l(i)$ marked points. Since in the following we will only be working with the intersection patterns of these curves, we will identify curves with their marked points, so that they are finite (multi-)sets.
				\end{remark}
				We mention one particularly useful construction, which produces a smoothing over the \textit{Artin component} of $\mathcal B(X)$, called the \textit{Scott deformation} of $X$.
				\begin{theorem}[{See \cite[Proposition 1.10]{de1998deformation}, \cite[Proposition 4.1]{plamenevskaya2023unexpected}}]\label{thm:scott}
					Let $(C,o)$ be an isolated curve singularity of multiplicity $m$. There is a 1-parameter deformation, the general fiber of which only has singularities of the following two types:
					\begin{itemize}
						\item singularities of $C$ after blowing up $o$
						\item an $m$-tuple point.
					\end{itemize}
				\end{theorem}
				Given a sandwiched singularity using Theorem \ref{thm:picdef} and iterating Theorem~\ref{thm:scott}, we obtain a picture deformation of $X$, the Milnor fiber of which is isomorphic to the resolution of $X$. This will be useful when we compare \QHD smoothings to the resolution of the singularity in Section~\ref{sect:latticemb}.\medskip
				
				One can gather all of the incidence information of this combinatorial setup into a $|\{\text{points}\}|\times|\{\text{curves}\}|$ matrix $I:=(\iota_{mn})$ where $\iota_{mn}=1$ if point $m$ is on curve $n$, and $0$ otherwise.
				\begin{theorem}[{\cite[Theorem 5.2]{de1998deformation}}]\label{exseq}
					There is an exact sequence 
					\begin{equation}
						0\to H_2(F)\to\mathbb Z\langle p\in C_t:m_p\neq 0\rangle\xrightarrow I\mathbb Z\langle(C_t)_i\rangle\to H_1(F)\to 0
					\end{equation}
					where $F$ is the Milnor fiber of the picture deformation.
				\end{theorem}
				Using \cite[Theorem 2]{greuel1981topology}, which states that for any smoothing of a normal surface singularity, the Milnor fiber $F$ satisfies $b_1(F)=0$, we get
				\begin{corollary}\label{muzero}
					$$\operatorname{rk} H_2(F)=:\mu(F)=\#\{\text{points}\}-\#\{\text{curves}\}.$$
				\end{corollary}
				Finally, we mention a method to recover the intersection lattice of the Milnor fiber of the smoothing from a picture deformation.
				\begin{theorem}[{\cite[Theorems 5.5, 5.11]{de1998deformation}}]\label{thm:int-embed}
					The intersection form on $H_2(F)\subset\mathbb Z\langle p\in C_t:m_p\neq 0\rangle$ is the restriction of the Euclidean negative definite inner product on the latter space.
					Moreover, the canonical class is represented by $(1,1,\dots,1)$ in this basis.
				\end{theorem}
				
				\section{Combinatorial smoothings}\label{sect:theorem}
				\subsection{The minimal rational case}Throughout this section, we assume $\Gamma$ to be a minimal rational graph with a given sandwich presentation where each component of the corresponding decorated curve is smooth.
				
				\begin{definition}
					In a tree graph $\Gamma$ we denote the unique path between two vertices $v,w\in\Gamma$ by $p(v,w):=(v=v_1,\dots,v_n=w)\subset V(\Gamma)$.
				\end{definition}
				\begin{definition}\label{def:sandwich presentation}
					Given a minimal rational graph $\Gamma$, choose a vertex $v\in\Gamma$. Add $-(deg(v)+e(v))-1\geq0$ many $-1$ leaves to it, and $-(deg(w)+e(w))$ many $-1$ leaves to every other vertex ($deg$ denotes the degree of the vertex in the graph $\Gamma$, and $e$ is the Euler number of the vertex). This larger graph $\tilde\Gamma$ is a \textit{sandwich presentation} of $\Gamma$ \textit{with smooth branches}, where \textit{the blowdown ends at $v$}.
				\end{definition}
				We will be only interested in the incidence structure of the smoothings. To this end, we identify the components of the deformed curve $(C)_i$ with the support of their decoration divisor $(l)_i$. We make this identification concrete with the following definition.
				\begin{definition}\label{def:smoothsmoothing}
					Given a sandwich presentation of $\Gamma$ with smooth branches define a \textit{combinatorial smoothing} to be a tuple $(V,\mathcal C)$, where $V$ is a set of points, and $\mathcal C\subset\mathcal P(V)$ represents the curves corresponding to each $-1$ leaf of the presentation.
					We denote a curve on a vertex $w\in\Gamma$ by $C_w$, and $v$ denotes the last vertex of the blowdown.
					The curves have to abide by the following two rules:
					\begin{enumerate}
						\item $|C_w|=|p(w,v)|+1$
						\item $|C_w\cap C_z|=|p(w,v)\cap p(z,v)|$
					\end{enumerate}
				\end{definition}
				\begin{remark}
					Using the Noether multiplicity formula, it is clear that given a sandwiched singularity with smooth branches, this is the incidence structure of a picture deformation.
					
					Again, we note that curves are identified with the support of their decoration divisor, and so are finite sets.
					
				\end{remark}
				There are many different ways to give a sandwich presentation of $\Gamma$ with smooth curves, this corresponds to the choice of $v$ in Definition~\ref{def:sandwich presentation}, the vertex where we put one curve less than would be required to blow it down.
				\cite[Theorem 4.16.]{de1998deformation} tells us how to relate combinatorial smoothings under a change of this  last vertex  $v$.
				\begin{theorem}[{[loc. cit.]}]\label{switch}
					Given a sandwich presentation of $\Gamma$ with smooth branches, ending at a vertex $v$, a combinatorial smoothing $(V,\mathcal C)$ and another vertex $w$ with $deg(w)+e(w)<0$ one gets another sandwich presentation of $\Gamma$ with smooth branches, ending at $w$ by adding an extra $-1$ leaf to $v$, and removing one from $w$.
					There is a combinatorial smoothing of this graph $(\tilde V,\tilde{\mathcal C})$ obtained as follows: 
					\begin{enumerate}
						\item $\tilde V:=V$
						\item choose an arbitrary $C_w$ curve to remove
						\item the new curve $\tilde C_v:=C_w$
						\item for $C\neq C_w$ and $p\in V$ we have $\tilde C= C_w\triangle C$ 
					\end{enumerate}
				\end{theorem}
				\begin{remark}
					Switching from $C_v$ to $C_w$, and then back to $C_v$, gives back the original smoothing.
				\end{remark}
				\begin{definition}
					We put a partial ordering on the graph: $x>y$ if $y\in p(x,v)$ where $v$ is the last vertex to be blown down.
				\end{definition}
				\begin{remark}
					From the definition it is clear, that if $x\leq y$, then $|C_x\cap C_y|=|C_x|-1$, or stated differently $\exists! p\in C_x:C_x\setminus\{p\}\subset C_y$.
					On the other hand, if $x\not\gtrless y$ and $z\geq y$, then $|C_x\cap C_y|=|C_x\cap C_z|$.
					These basic properties will be used many times in the following.
				\end{remark}
				\begin{definition}\label{dfn:redtrip}
					A \textit{reducing triple} consists of three (not necessarily distinct) vertices $v\leq w\leq z$, subject to the following conditions:
					\begin{enumerate}
						\item\label{dfn:redtrip-int} $\{deg(v),deg(w),deg(z)\}\subset\{1,2\}$ or $v=w=z$
						\item\label{dfn:redtrip-curves} there is one curve on $w$ and at least one on $v$ and $z$ 
						\item no vertex in $p(v,z)\setminus\{v,w,z\}$ has any curves
						\item $\Gamma\setminus p(v,z)$ and $p(v,z)$ only connect with edges in $\Gamma$ at $v$ and $z$
					\end{enumerate}
					(\ref{dfn:redtrip-curves}) is to be understood additively, i.e. if say $v=w$, then it should have at least two curves (one for $w$ and at least one more for $v$).
				\end{definition}
				\begin{definition}\label{def:qp}
					For a reducing triple $v\leq w\leq z$, we let the blowdown end at $v$ and pick a distinguished curve on $z$ and denote it $C_z$, the curve on $w$ is denoted $C_w$.
					We name $C_w\setminus C_z=:\{Q\},\ C_w\cap C_z=:P,\ C_z\setminus P=:Q'$.
					Similarly, we will denote curves on a vertex $x$ by $C_x^1, C_x^2$, and so on.
				\end{definition}
				\subsection{The reduction algorithm}\label{subsec:algo}
				Next, we prove a few simple statements in the context of a given reducing triple.
				From now on, we also suppose that every node satisfies $deg(v)+e(v)\leq -2$.
				
				\begin{lemma}\label{lem:qpropagate}
					Let $w\neq x\in\Gamma$ be a vertex. If $Q\in C_x$, then $Q\in C_a$ for all $x\leq a$.
				\end{lemma}
				\begin{proof}
					Since $Q\in C_x$, either $Q'\subset C_x$ if $x\geq z$, or $\exists!q'\in Q':q'\in C_x$ if $x\not\geq z$.
					Now supposing  $Q\not\in C_a$ implies $C_x\setminus\{Q\}\subset C_a$ since $x\leq a$. In the first case this implies that $P\subset C_a$ by intersection count with $C_w$ and since $Q'\subset C_x\setminus\{Q\}$ we get that $C_z=P\cup Q'\subset C_a$, a contradiction.
					Similarly in the second case $q'\in C_a$ implies, that $Q\in C_a$, contradiction.
				\end{proof}
				\begin{corollary}
					If $Q\in C_x,\ Q\not\in C_y$, then $x\neq y$ unless $w=z=x=y$.\hfill\qedsymbol
				\end{corollary}
				\begin{remark}
					By this corollary, we can say that a \textit{vertex $v\in\Gamma$ contains $Q$} or not, dependent on if its curves do.
				\end{remark}
				\begin{lemma}
					If $a,b\in\Gamma$ are in different components of $\Gamma\setminus \{v\}$ and $Q\in C_a$, then $Q\not\in C_b$.
				\end{lemma}
				\begin{proof}
					The $deg(v)=1$ case is vacuous.
					Secondly, if $deg(v)=2$, $Q\in C_a\cap C_b$ suppose $z\leq a$, now $b$ being in a different component means, that $C_b$ has to intersect $C_a$ once, but from $Q\in C_b$, we get that there is also a $q'\in Q'$ which is in $C_b$ since it has to intersect $C_z$ once.
					$z\leq a$ implies that $Q'\subset C_a$, thus $|C_a\cap C_b|\geq 2$, a contradiction.
					
					In the third case, when the reducing triple is $v,v,v$, $Q'$ consists of a single point, and the same argument shows that $Q\in C_a\cap C_b$ implies $Q'\in C_a\cap C_b$, but $|C_a\cap C_b|=1$ by assumption.
				\end{proof}
				
				\begin{lemma}\label{oneway}
					Let $(w\neq)x<a,b$ with $a\not\gtrless b$ in two different components of $\Gamma\setminus\{x\}$.
					If $Q\not\in C_x^1,C_x^2$, then $Q$ cannot be in both $C_a$ and $C_b$.
				\end{lemma}
				\begin{proof}
					Assume $x>z$ and let $n+1=|C_x^i|$, if $Q\in C_a$, then there is a point $q'\in Q'$ which is in $C_a$ and a point $p\in P$ with $p\not\in C_a$.
					Call $C_x^i\setminus C_x^j=\{p_i\}$.
					Since $C_a$ does not contain a point of $P$ and contains $Q$, it also contains both $p_i$'s.
					This means, that if $Q\in C_b$ were true, then $C_a$ and $C_b$ would intersect in at least $n-2$ points of $C_x^1\cap C_x^2$, in $Q$ and $p_1,p_2$ for a total of at least $n+1$ points, a contradiction, since they have to intersect $n$ times by assumption.
					
					If $x$ is in another component of $\Gamma\setminus\{v\}$, we argue similarly.
					The $C_x^i$ intersect $C_w, C_z$ in a single point, which can be in $C_x^1\cap C_x^2$, or the $p_i$'s. In either case, $C_a, C_b$ have to intersect in the remaining at least $n$ points of $C_x^1\cup C_x^2$, and in $Q$, a contradiction.
				\end{proof}

				We can extend the notion of reducing triple to the case where we replace (\ref{dfn:redtrip-int}) of Definition~\ref{dfn:redtrip} with $v=w$, $deg(v)=deg(w)<3$ and $deg(z)>2$.
				\begin{lemma}\label{leafred}
					With the modified assumption on the reducing triple, we have, that $C_z^1\cap C_z^2\cap C_v=\{P\}$.
				\end{lemma}
				\begin{proof}
					Let the blowdown end at $v$.
					We choose one of the curves on $z$ to be $C_z:=C_z^1$, and denote $C_z\cap C_v=:\{P\}$, $C_v\setminus\{P\}=:\{Q\}$ and $C_z\setminus\{Q\}=:Q'$ 
					as usual. Furthermore, assume that $Q\in C_z^2$.
					Pick $x>z$ with a curve.
					Without loss of generality, assume that $P\in C_x$, then $Q'\subset C_x$, since it missed $Q\in C_z^2$, and it cannot intersect $C_z^2$ enough times otherwise. This implies $C_z\subset C_x$, contradiction.
				\end{proof}
				This tells us which point of $C_v$ to label $Q$, now we only need to check that only one component of $\Gamma\setminus\{z\}$ and $\Gamma\setminus\{v\}$ can contain curves with $Q$, afterwards Lemma~\ref{lem:qpropagate} and Lemma~\ref{oneway} are applicable.
				\begin{lemma}
					If $x\not\gtrless y$ are in two different components of $\Gamma\setminus\{z\}$, then $x,y$ cannot both contain $Q$.
				\end{lemma}
				\begin{proof}
					The assumption $Q\in C_x\cap C_y$ implies that $C_x$ and $C_y$ intersect $n+1$ times, since they have to contain $(C_z^1\cup C_z^2)\setminus\{P\}$, contradiction.
				\end{proof}
				\begin{lemma}
					Consider $x>z$ and $y$ in another component of $\Gamma\setminus\{v\}$, then $x,y$ cannot both contain $Q$.
				\end{lemma}
				\begin{proof}
					By $Q\in C_x$, we get that $Q'\subset C_x$, and by the assumption on $y$, we get that $C_y\cap Q'=\emptyset$, so $C_y\cap C_z=\emptyset$, contradiction.
				\end{proof}
				Consider a graph $\Gamma$ and a $\mu=0$ combinatorial smoothing of it $(V,\mathcal C)$.
				In this case, the map $I$ of Theorem \ref{exseq} is injective, in particular invertible, so by \cite[5.12-13]{de1998deformation}, such a combinatorial smoothing is a qG-smoothing (see \cite{kollar1990flips}) and, in particular, has no free points.
				This guarantees us that there will be a curve besides $C_w$ containing $Q$.
				The preceding lemmas tell us that there is an edge of the graph that separates the vertices containing $Q$ from the vertices not containing it. We call such an edge a \textit{separating edge}.
				
				Now, if there is a reducing triple in $\Gamma$, we can construct a smaller graph and a combinatorial smoothing for it by contracting a separating edge, deleting $Q$ from $V$ and $C_w$ from $\mathcal C$.
				This can be iterated until the graph has no further reducing triples.
				A graph with no reducing triples consists of nodes with two curves and other vertices supporting at most one curve, such that every component of $\Gamma\setminus\{\text{nodes}\}$ has at most two curves total.
				\begin{definition}
					We call such a graph \emph{reduced}.
				\end{definition}
				\begin{remark}\label{rmk:compts}
					Note that during reduction, the number of nodes, leaves, and components of $\Gamma\setminus\{\text{nodes}\}$ can only decrease.
				\end{remark}
				
				\begin{definition}
					Given a sandwich graph $\Gamma$  we define $$\delta(\Gamma):=\left(\sum_{v\in\Gamma}-e(v)-d(v)\right)-1-|V(\Gamma)|.$$
				\end{definition}
				\begin{remark}
					Note, that given $(\mathcal C,V)$, a combinatorial smoothing of $\Gamma$ with smooth branches  we see $\delta=|\mathcal C|-|V(\Gamma)|$.
					Notice also, that $\delta(\Gamma)=1-\sum_\Gamma(e(v)+3)$, and compare with Wahl's formula in \cite[Theorem 8.1]{wahl2011rational}, which states, that the dimension of a \QHD smoothing component is given by $h^1(S)-\sum_\Gamma(e(v)+3)$, where S is the sheaf of \emph{logarithmic vector fields}, the kernel of the map $T_X\to\oplus \nu_v\to 0$.
					
					For graphs in $\mathcal A,\mathcal B,\mathcal C$, this quantity is always $3-s$, if the graph only has degree $3$ nodes, and $2-s$ if it has a degree 4 node, where $s:=\#\{v\in\Gamma:deg(v)>2\}$.
				\end{remark}
				The above proposed algorithm leaves this quantity unchanged. 
				\begin{lemma}\label{lem:finite}
					After reducing a graph $\Gamma\in\ABCt$ with $s$ many nodes, the reduced graph $\gamma$ satisfies $|\gamma|\leq 7s-2$.
					Similarly for a graph $\Gamma\in\ABCf$ we have $ |\gamma|\leq 7s+1$.
				\end{lemma}
				\begin{proof}
					In the first case $\Gamma\setminus\{\text{nodes}\}$ has at most $2s+1$ components, so $\gamma$ has at most this many.
					
					This implies, that $\gamma$ has at most $2s+2(2s+1)-1=6s+1$ curves, writing $3-s\leq 6s+1-|\gamma|$ and rearranging gives the result.
					
					For the case where $\Gamma$ has a degree 4 node, we compute similarly: $2s+2(2s+2)-1=6s+3$ is the upper bound on the curves, $2-s\leq 6s+3-|\gamma|$ and the statement follows.
				\end{proof}
				
				Using Lemma~\ref{lem:finite} and computer checking we obtain special cases of \cite[Theorem 1.4, 1.6]{bhupal2011weighted} and \cite[Theorem 8.6]{wahl2011rational}:
				\begin{corollary}
					Let $\Gamma\in\mathcal A\cup\mathcal B\cup\mathcal C$ be a graph with one node of degree $3$ (resp. $4$) with framing at most $-5$ ($-6$), then $\Gamma$ cannot admit a \QHD smoothing.
				\end{corollary}
				\begin{proof}
					Suppose the graph has a \QHD smoothing; apply the reduction algorithm to its incidence structure. The reduced graph still has $\delta=2$ $(1)$, with $e(v)=-2$ if $deg(v)=1$, $e(v)\in\{-2,-3\}$ if $deg(v)=2$, $e(v)=-deg(v)-2$ otherwise.
					A linear graph under these framing restrictions has $\delta\leq -1$, an immediate contradiction.
					Thus, the graphs have a node and at least $3$ leaves. This means that they have $\delta\leq 0$, which is a contradiction.
				\end{proof}
				\begin{corollary}
					Let $\Gamma\in\mathcal A\cup\mathcal B\cup\mathcal C$ be a graph with two nodes of degree 3 and framing at most $-5$, then $\Gamma$ cannot admit a \QHD smoothing.
				\end{corollary}
				\begin{proof}
					Suppose $\Gamma$ admits a \QHD smoothing, choosing a sandwich presentation for $\Gamma$ with smooth curvettas (since it is assumed to be minimal rational), we get a combinatorial \QHD smoothing, to which we apply the reduction algorithm, after which we obtain a reduced graph $\gamma$.
					
					By Lemma~\ref{leafred}, we have the same framing possibilities as before.
					Since $\delta$ stays constant, we know that the reduced graph has $\delta=1$, since this is true for the starting graphs. By the previous corollary, graphs with a single node have $\delta\leq 0$, so all possible graphs have two nodes. Checking all trees $\gamma$ with $|\gamma|\leq 12$ under the constraints derived thus far, we see that there is no reduced graph with non-square determinant.
				\end{proof}

				Furthermore, we can extend this to graphs with a degree 4 and a degree 3 node, and to the cases with 3 and 4 nodes as well. In these cases, the determinant can be a square, so we need to rely on Theorem~\ref{thm:zksq}, proved below, which states that for a graph admitting a combinatorial \QHD smoothing, one has $Z_K^2+|\Gamma|=0$.
				Unfortunately the computation of the $d$-invariant\footnote{i.e. $(Z_K^2+|\Gamma|)/4$, see e.g. \cite{nemethi2022normal}} for general reduced graphs remains elusive.
				\begin{theorem}
					Let $\Gamma\in\mathcal A\cup\mathcal B\cup\mathcal C$ be a graph with  one node of valency $4$ and framing at most $-6$, and at most $2$ nodes of valency $3$ and framing at most $-5$. Then $\Gamma$ cannot admit a \QHD smoothing.\hfill\qed
				\end{theorem}
				
				\begin{theorem}
					Let $\Gamma\in\mathcal A\cup\mathcal B\cup\mathcal C$ be a graph with at most four nodes of valency $3$, and framing at most $-5$, then $\Gamma$ cannot admit a \QHD smoothing.\hfill\qed
				\end{theorem}

\section{Lattice embeddings}\label{sect:latticemb}
In this brief section, we wish to prove the following:
\begin{theorem}\label{thm:zksq}
	If a sandwiched graph $\Gamma$ admits a combinatorial rational homology disk smoothing, then the anticanonical cycle of the graph $Z_K$ satisfies $Z_K^2+|\Gamma|=0$.
\end{theorem}
\begin{proof}
	Consider a sandwiched graph $\Gamma$ with a fixed sandwich presentation $\tilde\Gamma$, which admits a combinatorial \QHD smoothing.
	Following \cite{de1998deformation}, we let $P'=\Z^l$ denote the free abelian group generated by the points of the combinatorial \QHD smoothing and endow it with the trivial negative definite intersection form.
	Similarly, let $L=\Z^l$ denote the free abelian group generated by the set of curvettas of the sandwich presentation $\tilde\Gamma$.
	Consider the map $J: P'\to L$ where $p\mapsto \sum_{p\in L_i} L_i$, i.e., the incidence matrix of the combinatorial smoothing.
	Since $\rk P'=\rk L$, by \cite[Corollary 5.10]{de1998deformation}, the matrix of this map in the natural bases is invertible over $\mathbb Q$.
	
	Consider also the incidence map of the Scott deformation $I: P=\Z^{l+n}\to L$. As mentioned in \cite[Corollary 1.11]{de1998deformation} and \cite[Subsection 4.2]{plamenevskaya2023unexpected}, the Milnor fiber of this deformation is diffeomorphic to the resolution of the singularity; thus, by Theorem \ref{thm:int-embed}, we have $Q_\Gamma\hookrightarrow P$.
	
	Next note that $II^*=JJ^*=Q_{artin}$ in the language of \cite{de1998deformation}, thus $J^{-1}II^*{J^*}^{-1}=(J^{-1}I)(J^{-1}I)^*=id$, which means that the rows of the matrix $J^{-1}I$ form an orthonormal basis for $\ker(J^{-1}I)^\perp=\ker(I)^\perp\leq P\otimes\mathbb Q$.
	
	Let $1_\alpha=(1)_1^\alpha$. The product $I1_{l+n}=(|L_i|)_1^l$. Since the number of points on each curvetta is determined by the sandwich presentation $\tilde\Gamma$, the same is true for $J1_l$. Thus $J^{-1}I1_{l+n}=1_l$. This means that denoting the rows of $J^{-1}I$ by $f_i$ we obtain a basis where taking $K=\sum_{p\in P}p$ we have:
	\begin{equation}\label{eq:K im}
		(J^{-1}I)^*(J^{-1}I)K=\sum f_i
	\end{equation}
	
	Up until now, we were concerned with $\ker(I)^\perp$. We have $P\otimes \Q=(\Q^{l+n},\langle-1\rangle^{l+n})$, and an embedded $(\Q^l,\langle-1\rangle^l)$.
	By Witt's cancellation theorem (\cite[4.4]{milnor1973symmetric}) this means that the orthogonal complement generated by $\ker(J^{-1}I)$ is isomorphic to $(\Q^n,\langle-1\rangle^n)$.
	
	By Equation \ref{eq:K im} the orthogonal projection of $K$ onto $\ker(I)^\perp$ is equal to $\sum f_i$, and so has square $-l$, which implies that the orthogonal projection to $\ker(J^{-1}I)$ will have square $-n$. From Theorem \ref{thm:int-embed}, this projection represents the anti-canonical cycle of $\Gamma$; thus we get that $Z_K^2=-n$.
\end{proof}

\begin{proposition}
	If a sandwiched graph $\Gamma$ admits a combinatorial rational homology disk smoothing, then $Q_\Gamma\subset U$ for some unimodular lattice $U$, with $\rk\ U=|\Gamma|$.
\end{proposition}
\begin{proof}
	Let $I: P\to L$ denote the incidence matrix of the Scott deformation, and $J:P'\to L$ denote the incidence matrix of a combinatorial \QHD deformation, as previously. By Figure \ref{fig:diagram} (\cite[Diagram 5.9]{de1998deformation}) we get that $I$ is surjective and $J$ is injective. In particular $H^2(F_J)\leq H_1(\partial F_J)$ is a subgroup of the discriminant group of $Q_\Gamma$ (which is isomorphic to $H_1(\partial F)$ of size $\sqrt{|H_1(\partial F)|}$).
	By commutativity of the diagram and the fact that $P$ is endowed with the trivial negative definite intersection form, this subgroup is isotropic, which by \cite[1.7.1]{looijenga1986quadratic} gives an integral unimodular overlattice $U$ for $Q_\Gamma$ in $Q_\Gamma^*$.
	\begin{figure}[h!]
		\centering
		
		\adjustbox{scale=0.9}{
			\begin{tikzcd}
				&                            & 0 \arrow[d]                    & 0 \arrow[d]                           &                              &   \\
				& 0 \arrow[d] \arrow[r]      & L^* \arrow[d, "f^*"] \arrow[r] & L^* \arrow[d, "f\circ f^*"] \arrow[r] & 0 \arrow[d]                  &   \\
				0 \arrow[r] & \ker f \arrow[r] \arrow[d] & P \arrow[r, "f"] \arrow[d]     & L \arrow[r] \arrow[d]                 & co\ker f \arrow[r] \arrow[d] & 0 \\
				0 \arrow[r] & H_2(F_f) \arrow[r] \arrow[d] & H^2(F_f) \arrow[r] \arrow[d]     & H_1(\partial F_f) \arrow[r] \arrow[d]   & H_1(F_f) \arrow[r] \arrow[d]   & 0 \\
				& 0                          & 0                              & 0                                     & 0                            &  
			\end{tikzcd}
		}
		
		\caption{$f$ denotes the incidence matrix of a combinatorial smoothing, and $F_f$ the Milnor fiber (\cite[Diagram 5.9]{de1998deformation})}
		\label{fig:diagram}
		
	\end{figure}
\end{proof}

\section{Some Graphs Admitting a combinatorial \QHD smoothing}\label{sect:examples}
\begin{example}
	Let us denote by $fpp(n)$ the star-shaped graph with $n^2+n+1$ arms, each of length $n-1$, with the node having framing $-n^2-n-2$, and all other vertices $-2$.
	A combinatorial \QHD smoothing of this graph is the finite projective plane of order $n$. $\det(fpp(n))=n^{(n+1)n}(n+1)^2$, and they all have $Z_K^2+n=0$.
	
	By modifying the finite projective plane configuration, we obtain further graphs with combinatorial \QHD smoothings.
	Pick a point $x$ and all lines through it. Add further points $a_1,\dots,a_l$, and curves $\{x,a_i\}$ to the configuration, and for all lines that do not contain $x$, add all of the $a_i$. This is a \QHD smoothing for the graph we denote $fpp(n)_l$.
\end{example}
\begin{example}\label{ex:cluster}
	Consider $k$ "\textit{clusters}" consisting of $n$ points each: $\{A_i\}_1^k$ with $|A_i|=n$. Define sets of curves as follows: $C^i_j=\cup_{l\neq i} A_l\cup \{a_j\}$ where $a_j\in A_i$. This is a combinatorial \QHD smoothing of a graph, which we denote $Cl(k,n)$.

	\vspace{-1cm}
	\begin{figure}[h!]
		\includegraphics{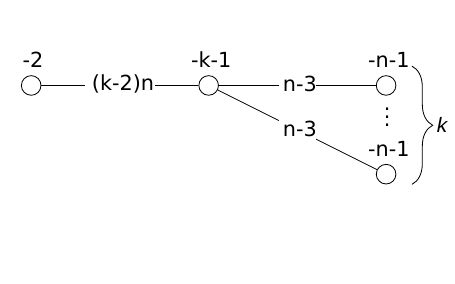}
		\vspace{-1.5cm}
		\label{fig:clustergraph}
		\caption{The graph $Cl(k,n)$. The number $x$ on an edge represents a path of $x$ many $-2$ vertices.}
	\end{figure}
\end{example}

\begin{example}
	We define cluster extension of the graphs $Cl(k,n)$. Here, consider an additional set $B$, with $|B|=b$, and add its points to all existing curves.
	Furthermore add new curves $D_i=\cup A_l\cup \{b_i\}$. The construction corresponds to a combinatorial smoothing of a graph, if $b\geq n\geq 2$.
	
	\begin{figure}[h!]
		\includegraphics{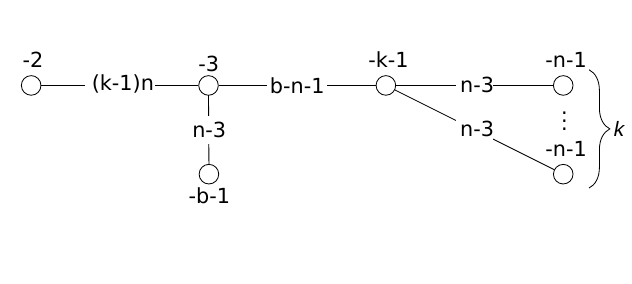}
		\centering \vspace{10pt}
		\vspace{-1.5cm}
		
		\caption{The graph we obtain after cluster extending $Cl(k,n)$.}
		\label{fig:clustergraph clusterext}
	\end{figure}
\end{example}

\begin{example}
	
	Another construction one can make is the star extension of $Cl(k,n)$.
	Add a new point $x$ and a new set $B$ of size $b$. Add the elements of $B$ to the existing curves, and define new curves $S_i=\{x\}\cup\{b_i\}$  and  $F=\cup_1^k A_i\cup\{x\}$.
	\begin{figure}[ht]
		\includegraphics{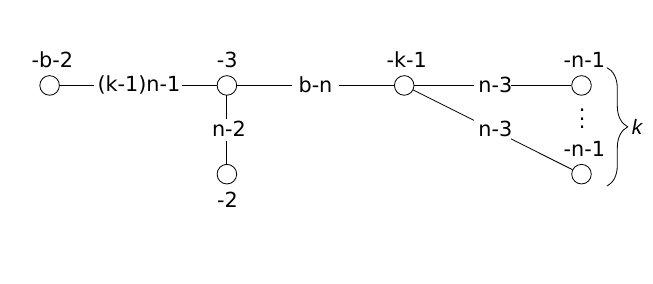}
		\centering\vspace{10pt}
		\vspace{-1.5cm}
		\caption{The graph we obtain after star extending $Cl(k,n)$.}
		\label{fig:clustergraph clusterext2}
	\end{figure}

\end{example}

\begin{example}
	Using the complement idea, we generalize the family of Wahl \cite[5.9.2]{wahl1981smoothings}. In this case we consider 3 sets $A, B, C$ of sizes $|a|,|b|,|c|$, the curves are $A\cup \{b_i\},\ B\cup \{c_i\},\ C\cup \{a_i\}$ with $a_i\in A, b_i\in B, c_i\in C$, or one can take the complements of the singletons inside their containing sets. We denote this choice by flipping the sign of the set, that the corresponding curves contain, i.e. $a,-b, c$ uses the curves $A\cup\{b_i\},B\cup(C\setminus\{c_i\}), C\cup\{a_i\}$.
	This gives combinatorial smoothings for the graphs we denote $t(a,b,c)$.
	If all parameters are positive, we get Wahl's family $\mathcal W$ (Figure \ref{fig:WMN} left), if one is negative, we get the $\mathcal N$ family (Figure \ref{fig:WMN} right). In the case $t(a,-b,-c)$ for the curves to be the smoothing of a graph we need either $a=2$ or $c=2$ and get the subfamily of N with $p=0$ and the subfamily of $W$ with $q=0$.
	If all three are negative with $|a|>|b|=|c|$ we get the novel graph of Figure \ref{fig:smalltneg}.
	
	\begin{figure}[h!]
		\includegraphics{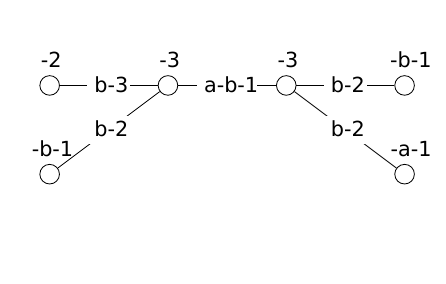}
			\vspace{-1.5cm}
			\caption{The graph $t(-a,-b,-b)$.}
			\label{fig:smalltneg}
		\end{figure}
		
	\end{example}

\section{Weighted homogeneous singularities -- stars}\label{sect:bpst}

This section presents a new proof of the main theorem of \cite{bhupal2011weighted}, namely the classification of weighted homogeneous singularities admitting a \QHD smoothing. From Theorem~\ref{thm:existence}, we only need to check the star-shaped elements of $\ABC$. In contrast to Section~\ref{sect:theorem}, not all graphs will be minimal rational; most curvettas will be ordinary cusps.

\subsection{The degree 3 case}
Let us begin with some generalities, which will be used multiple times in the following.
We will be dealing with star-shaped graphs with three arms, where two of the arms have length $1$ and decorations $-a$ and $-b$, respectively.
We will always put $\max\{a-2,0\}$ many $-1$ framed  vertices on the $-a$ vertex and $\max\{b-3,0\}$ many $-1$ framed vertices to the $-b$ vertex. Denote curvettas on these vertices by $S_i$ and $L_i$, respectively.
The third "long" arm will be blown down starting from the end and decorated as necessary.
Take an arbitrary order on the $-1$ leaves on every vertex, and denote the $j$th curvetta of the $i$th vertex by $C_i^j$.

\begin{figure}[h!]
	\includegraphics{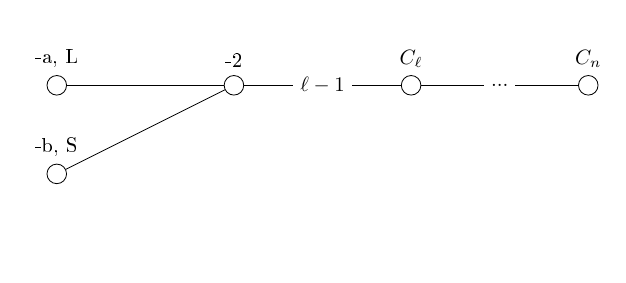}
	%
	\vspace{-1cm}
	\caption{The general case graph with the curves indicated.}
	\label{fig:screenshot001}
\end{figure}

Since the node has framing at most $-2$, we know that the $C_i^j$ will be star-shaped and the $L_i, S_i$ will be smooth.
We provide their intersections in Table \ref{intersect}.
\begin{table}[h!]
	\begin{subtable}[h]{0.45\textwidth}
		\begin{tabular}{|c|c|c|c|}
			\hline
			Curve & $L_i$ & $S_i$ & $C_i^j$ \\
			\hline
			\# of points & 2 & 3 & 5+i \\
			\hline
		\end{tabular}
		\caption{Point data}
		\label{points}
	\end{subtable}\hfill
	\begin{subtable}[h]{0.45\textwidth}
		\begin{tabular}{|c|c|c|c|}
			\hline
			$(\cdot,\cdot)$ & $L_a$ & $S_a$ & $C_a^b$ \\
			\hline
			$L_i$ & 1 & 1 & 2 \\
			\hline
			$S_i$ &  & 2 & 3 \\
			\hline
			$C_i^j$ &  &  & $6+\min\{i,a\}$ \\
			\hline
		\end{tabular}
		\caption{Intersection data}
		\label{intersect}
	\end{subtable}
	
	\caption{The combinatorial information in our setup}
\end{table}

The multiplicity sequence of these cusps are easily computed to be $(2,1,\dots,1)$, so $C_i^j$ must all have one double (2-tuple) point and further simple points (i.e., at all but one point the curve is locally irreducible).
For a cusp curve $C$, we define $X(C)$ to be the point in $P$ which it contains with multiplicity $2$.
From this, it is simple to compute how many points each curve has to contain and their intersection multiplicities, this is collected in Table \ref{points}. We denote the points where these curves intersect by $\{p_i:1\leq i\leq 4+n\}$.

In this section, $n$ will always denote the number of vertices on the "long" arm (on which the $C_i^j$ are located).
From the construction of the $\mathcal A, \mathcal B, \mathcal C$ classes, it is easy to see by induction that there will always be $4+n$ many curves in total in the sandwich presentation described above.
We will mostly be interested in the smooth curves, the curves at the very end $C_n^\alpha$, and the curve closest to the node.
We let $\ell$ denote the index where there exists $C_\ell^1$, and for every $k<\ell$, there is no $C_k^1$.

\begin{remark}
	As before, we work with the intersection pattern of the given sandwich presentation of the graph, and by a small abuse of notation, we identify curves with the intersection points they contain.
\end{remark}

We note some simple facts about these curves:
\begin{lemma}\label{cusplemma}
	Consider $C_i^j$ and $C_a^b$ for $i\leq a$.
	If $X(C_i^j)\neq X(C_a^b)$, then $C_i^j\subset C_a^b$ (as sets, without multiplicities).
	If $X(C_i^j)= X(C_a^b)$, then there is a single $p\in C_i^j$ with $p\not\in C_a^b$.
\end{lemma}
\begin{proof}
	If their double points differ, then $p_d:=X(C_a^b)\in C_i^j$, otherwise, $(C_a^b,C_i^j)\leq 5+i<6+i$, and if $p_d\in C_i^j$, then it has to go through every point for the intersection to reach $6+i$.
	
	On the other hand, if their double points coincide, then there are a further $3+i$ simple points to make $2+i$ intersections. The claim follows.
\end{proof}
\begin{corollary}\label{Cnlemma}
	The $C_n^\alpha$ curves all have distinct double points, and all other $C_i^j$ curves must pass through these double points with multiplicity 1.
	The $L_i$ and $S_i$ cannot go through these points.
\end{corollary}
\begin{proof}
	Since the $C_n^\alpha$ curves contain $5+n$ points with multiplicity, and there are $4+n$ points in total, they go through all points, i.e., it is impossible for the second case of Lemma \ref{cusplemma} to occur.
	
	Both $L_i$ and $S_i$ would intersect the $C_n^\alpha$ curves too many times if they went through the double points.
\end{proof}
We state the following for emphasis:
\begin{corollary}\label{doublelemma}
	All $C_i^j$ must go through every other $C_a^b$'s double point.
\end{corollary}
\subsubsection{Graphs in $\mathcal C$}
We begin the case analysis with the class $\mathcal C$.

\begin{figure}[h!]
	\includegraphics{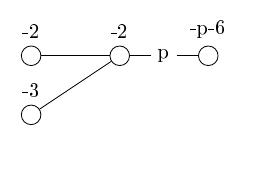}
	\vspace{-1cm}
	\caption{Graphs in $\mathcal C_6$ admitting a \QHD smoothing}
	\label{fig:C6}
\end{figure}

\begin{theorem}
	In $\mathcal C_6$, only the family depicted in Figure~\ref{fig:C6} can admit a \QHD smoothing.
\end{theorem}
\begin{proof}
	In this case, there are no $L_i, S_i$, only the cusps.
	The known case is $\ell=n$, we want to rule out any other possibility.
	It is easy to see by induction that there will be $6+\ell$ many $C_n^\alpha$ curves if $\ell<n$.
	By Corollary \ref{doublelemma}, $C_\ell^1$ would have to contain $2+6+\ell>5+\ell$ many points, a contradiction.
\end{proof}

\begin{figure}[h!]
	\includegraphics{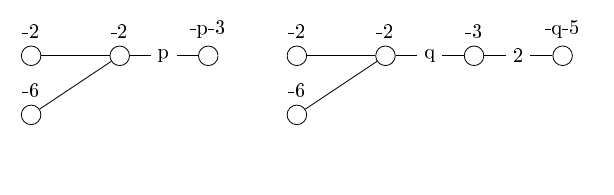}
	\vspace{-1cm}
	\caption{The graphs for the $C_3$ case, $p\geq0, q\geq -1$.}
	\label{fig:C3}
\end{figure}
\begin{theorem}
	In $\mathcal C_3$, only the graphs of Figure~\ref{fig:C3} can admit \QHD smoothings.
\end{theorem}
\begin{proof}
	In this case, we have three $L_i$'s in addition to the cusps.
	If $\ell=n$, we are done, so let $\ell<n$.
	Simple induction shows that there will be $3+ \ell$ many $C_n^\alpha$'s, so $C_\ell^1$ contains their double points, its double point and no further point.
	This means that the $L_i$ have to go through $C_\ell^1$'s double point to intersect it twice by Lemma \ref{Cnlemma}.
	
	If there would be another $C_k^\alpha$ with $\ell\leq k<n$, then it intersects $C_\ell^1$  least $7+ \ell$ times by Corollary \ref{doublelemma}, a contradiction.
	So there is $C_\ell^1$, and the $C_n^\alpha$ curves.
	This means that $3+\ell=4+n-3-1$, ergo $\ell=n-3$, and this is the other known case.

\end{proof}

\begin{figure}[h!]
	%
	\includegraphics{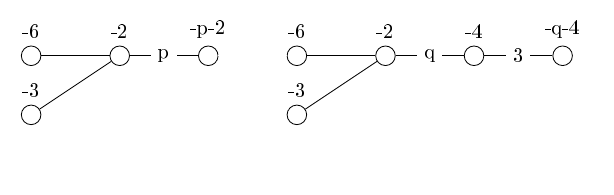}
	\includegraphics{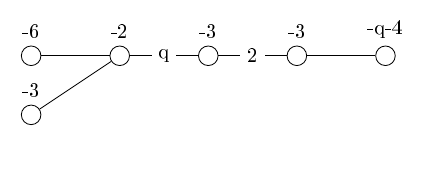}
	%
	\vspace{-1cm}
	\caption{The graphs for the $C_2$ case, $p\geq0, q\geq -1$.}
	\label{fig:C2}
\end{figure}
\begin{theorem}
	In $\mathcal C_2$, only the graphs of Figure~\ref{fig:C2} can admit a \QHD smoothing.
\end{theorem}

\begin{proof}
	In this case, there will be four $S_i$ curves and the cusps.
	Let $\ell<n$.
	There will be $2+ \ell$ many cusps on the last vertex.
	
	Counting the number of points of $C_\ell^1$ it has (besides the $C_n^\alpha$ double points)  $p_d:=X(C_\ell^1)$ and a further simple point denoted $p_s$.
	All of the $S_i$ have to go through $p_d$ and $p_s$ since  $(S_i,C_\ell^1)=3$.
	This also means that they have a further point each $p_{S_i}$.
	If there is a second $C_i^j$ for some $\ell\leq i<n$, then $X(C_i^j)=p_d$ or $p_s$.
	
	If $X(C_i^j)=p_d$, then $p_s$ will be the point $C_i^j$ misses by Lemma \ref{cusplemma}, and  $p_{S_i}\in C_i^j$.
	If there were a third $C_a^b$ curve with $i\leq a<n$, we see that $X(C_a^b)= p_d$ by Corollary~\ref{doublelemma}.
	By the same argument we get that $p_{S_i}\in C_a^b$ for $i=1,\dots,4$.
	
	Every subsequent curve has to meet these same requirements.
	Let $K:=\mathcal C\setminus\{S_i,C_n^\alpha,C_\ell^1\}$, with $|K|=k$. Let $\tilde K:=\cup K\setminus\{X(C_n^\alpha):1\leq \alpha\leq \ell+2\}\subset V$ with $|\tilde K|=\tilde k$. Counting the curves and points of the configuration, we get
	\begin{equation}\label{eq:C2}
		2+\ell+1+4+k=2+\ell+2+4+\tilde k=4+n,
	\end{equation}
	implying $k=1+\tilde k$.
	\begin{claim}
		$C_i^j=C_{\ell+3}^1$.
	\end{claim}
	
	\begin{proof}
		If $C_i^j$ would contain any further points, then the curves in $K$ restricted to the points of $\tilde K$ give a solution to the combinatorial smoothing problem for a linear graph, since $X(C)=p_d$ for all $C\in K$.
		This implies $k\leq\tilde k$, a contradiction.
	\end{proof}
	
	Any further curve would also go through $p_d$ and $p_{S_i}$, which means it would intersect $C_{\ell+3}^1$ $10+\ell>9+ \ell$ times, a contradiction. Substituting $k=1$ into Equation~\ref{eq:C2} we see that $\ell=n-4$, this is the third graph of Figure~\ref{fig:C2}.
	
	The other case plays out similarly.
	If $X(C_i^j)=p_s$, then every other curve has to have its double point at $p_s$ and go through $p_d$ to intersect $S_i$ and $C_\ell^1$ enough times.
	The same argument shows that $C_i^j=C_\ell^2$, there are no further curves and $\ell=n-4$.
	This is the second graph of Figure~\ref{fig:C2}.

	If there is only one cusp curve that is not on the last vertex, then we have $4+n=2+\ell+1+4$, ergo $\ell=n-3$.
	The $C_n^\alpha$'s require $2+n-3=n-1$ points, $C_\ell^1$ an additional $2$.
	We have only 3 points left, which is a contradiction since the $S_i$ need $4$ additional points besides the previously mentioned ones.
	
	Finally, if $\ell=n$, we get the first graph of Figure~\ref{fig:C2}.
\end{proof}
\subsubsection{Graphs in $\mathcal B$}
We continue with the $\mathcal B$ class.
\begin{figure}[h!]
	\includegraphics{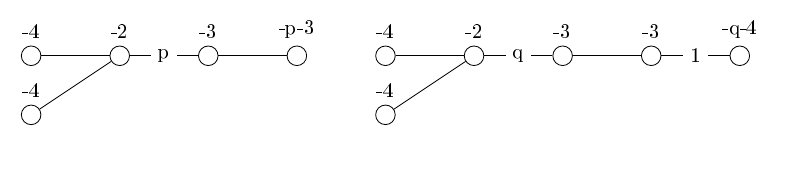}
	\vspace{-1cm}
	\caption{The graphs for the $B_2$ case, $p\geq0, q\geq -1$.}
	\label{fig:B2}
\end{figure}

\begin{theorem}
	In $\mathcal B_2$, only the graphs of Figure~\ref{fig:B2} can admit a \QHD smoothing.
\end{theorem}
\begin{proof}
	In this case, there are $L_1, S_1, S_2$ and the cusps. Suppose that not every blowup happened on the edge next to the leaf.
	In this case, the setup of the cusps is identical to the $\mathcal C_2$ case, the curve $C_\ell^1$ has (besides the $X(C_n^\alpha)$ points) a double point $p_d$ and one further simple point $p_s$. 
	$p_s,p_d\in S_i$ to intersect $C_\ell^1$ three times (and both have one further point $p_{S_i}$).
	$L_1$ has to go through $p_d$ to intersect $C_\ell^1$ twice, and thus it contains one further point $p_{L_1}$.
	
	We continue similarly to the $\mathcal C_2$ case.
	Any further curve can have its double point at either $p_d$ or $p_s$, and all of them have to be at the same point.
	Denote the curve with the next smallest index after $\ell$ by $C_i^j$.
	From here, the same argument provides that $C_i^j$ is $C_{\ell+1}^1$ in both cases (if its double point is $p_d$, it must contain $p_{S_1}$ and $p_{S_2}$; if its double point is $p_s$, it must contain $p_d$ and $p_{L_1}$).
	In both cases, there are no further curves because they would intersect $C_{\ell+1}^1$ too many times. This gives the second graph of Figure~\ref{fig:B2}.
	
	Lastly, if every blowup happens next to the leaf, we get the first graph of Figure~\ref{fig:B2}.
\end{proof}

\begin{figure}[h!]
	\includegraphics{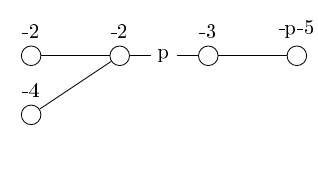}
	\vspace{-1cm}
	\caption{The graph for the $B_4$ case, $p\geq0$.}
	\label{fig:B4}
\end{figure}

\begin{theorem}
	In $\mathcal B_4$, only the graph of Figure~\ref{fig:B4} can admit a \QHD smoothing.
\end{theorem}
\begin{proof}
	In this case, we have $L_1$ and the cusps. Suppose not every blowup happened next to the leaf. Then a simple induction shows that there are $4+ \ell$ many $C_n^\alpha$ curves.
	This means $C_\ell^1$ goes through $6+ \ell$ points with multiplicity, an immediate contradiction.
\end{proof}

\subsubsection{Graphs in $\mathcal A$}
\begin{figure}[h!]
	\includegraphics{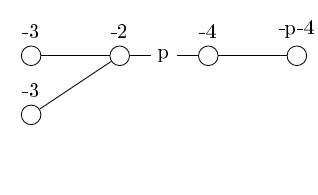}
	\vspace{-1cm}
	\caption{The graphs for the $A_3$ case, $p\geq0$.}
	\label{fig:A3}
\end{figure}
\begin{theorem}
	In $\mathcal A_3$, only the graph of Figure~\ref{fig:A3} can admit a \QHD smoothing.
\end{theorem}
\begin{proof}
	In this case, there are $S_1$ and the cusps. Suppose not every blowup happened next to the leaf.
	There will be $3+ \ell$ curves at the last vertex by induction.
	$C_\ell^1$ will go through the double points of the $C_n^\alpha$ curves, and it has its double point and no further points.
	By Corollary~\ref{Cnlemma}, $S_1$ can only intersect $C_\ell^1$ at $p_d$, so at most 2 times instead of the required 3, an immediate contradiction. Thus, we are left with the case when every blowup happens next to the leaf, depicted in Figure~\ref{fig:A3}.
\end{proof}

\subsection{The degree 4 case}
Most of the setup stays unchanged in this case; there will still be $4+n$ many curves, the end curves have to go through all points, etc.
In addition to the $L_i, S_i$, there will also be some
\textit{exceptional cusps} $\Gamma_i$.
They intersect the $L_i,S_i$ in the same manner as the $C_i^j$ curves do, and intersect the $C_i^j$ curves 6 times.
\begin{lemma}\label{gammalemma}
	$X(C_n^\alpha)\not\in\Gamma_i$ for any $i$.
\end{lemma}
\begin{proof}
	If this were not the case, they would intersect $7$ times.
\end{proof}
In the following, we will assume that not all blowups happen next to the leaf. As in the previous cases, a simple calculation shows, that $C_\ell^1$ has three points (besides the double points of the $C_n^\alpha$ curves of course): its double point denoted $p_d$, and two simple points $p_{s_1}$ and $p_{s_2}$.
We fix this notation in this section.
\begin{lemma}\label{2-3lemma}
	Every curve after $C_\ell^1$ can have its double point  at a maximum of two of the three points $p_d,p_{s_1},p_{s_2}$, i.e. $$\bigcup_{j>1} X(C_\ell^j)\cup\bigcup_{n>i>\ell,\ j>0} X(C_i^j)\subsetneq \{p_d,p_{s_1},p_{s_2}\}$$
\end{lemma}
\begin{proof}
	There are two cases to consider. The first curves have their double points at $p_d,p_{s_1}$ or $p_{s_1},p_{s_2}$ and both are similar.
	If there is a curve with double point at $p_d$ besides $C_\ell^1$, it has to miss one of the $p_{s_i}$, say $p_{s_2}$. This means no curve can have its double point at $p_{s_2}$ by Corollary \ref{doublelemma}.
	
	Symmetrically, if every subsequent curve has its double point at the $p_{s_i}$, no curve can have its double point at $p_d$ since it would have to miss one of the $p_{s_i}$ by Lemma \ref{cusplemma}.
\end{proof}
We state the following corollary of the proof for emphasis:
\begin{corollary}
	If $X(C_i^j)=p_d$, then $X(C_a^b)=p_d$ for any $a<i$.
\end{corollary}
\begin{proof}
	Such curves have to miss one of the $p_{s_i}$'s, and so a point of any curve with double point at any of these two points, which is a contradiction.
\end{proof}
\begin{lemma}\label{switchlemma}
	If every curve $B$ after $C_\ell^1$ until $C_i^j$ has $X(B)=p_d$ and the next curve $C$ satisfies $X(C)=p_{s_1}$, then all subsequent curves have their double points at $p_{s_1}$.
	The same statement is true if we change the order of $p_d,p_{s_1}$ or replace $p_d,p_{s_1}$ with $p_{s_1},p_{s_2}$.
\end{lemma}
\begin{proof}
	Let us denote a curve after $C$ by $C'$ with $X(C)\neq X(C')$.
	If $X(C_i^j)=p_d$, then by assumption we know, that $p_{s_1}\in C_i^j$ and $p_{s_2}\not\in C_i^j$; thus we only have to rule out $X(C')=p_d$.
	This is automatic by Lemma \ref{cusplemma}: since $p_d,p_{s_1},p_{s_2}\in C$, the curve $C'$ will also contain these points, and thus all points of $C_\ell^1$. This is a contradiction, since $X(C_\ell^1)=X(C')$.
	
	The case where $X(C_i^j)=p_{s_1}$ is simpler; if $X(C)=p_d$, then $p_{s_2}\not\in C$, a contradiction since $p_{s_2}\in C_i^j$.
	
	If $X(C)=p_{s_2}$, then $X(C')\neq p_d$, since it would miss a point of $C_i^j,C$.
	Since $X(C)\neq X(C')=X(C_i^j)$ we get that $C_i^j\subset C\subset C'$, a contradiction.
\end{proof}
The strategy of the proof follows the $\mathcal C_2,\mathcal B_2$ cases: we analyze the configuration of the existing $L_i,S_i,\Gamma_i$ curves in relation to $C_\ell^1$, then we take every other cusp, which is not a $C_n^\alpha$, and use them construct a solution to the combinatorial smoothing problem of a linear graph, i.e. a sequence of sets $A_i$ satisfying Definition \ref{def:smoothsmoothing}. 

This is done by observing, that if some cusps $C_i$ have $X(C_i)=X(C_{i'})$, then by "forgetting" the multiplicity of the double point the definition is already satisfied. Furthermore if there are larger (in the sense, that they correspond to vertices of the graph farther from the node) curves $D_i$ with $X(D_i)=X(D_{i'})\neq X(C_i)$ if we remove $X(C_i)$ from $D_i$ the definition is satisfied.
Furthermore we  delete points, which are contained in every curve, and identify points which we know must be either contained or missed together.

\begin{figure}[h]
	\includegraphics{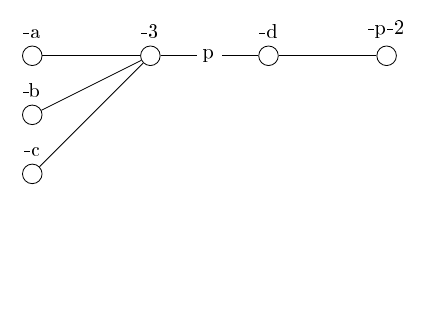}
	\vspace{-2cm}
	\caption{The valency 4 graphs, which admit a \QHD smoothing, where $[a,b,c;d]\in\{[3,3,3;4], [2,4,4;3],[2,3,6;2]\}$ and $p\geq 0$.}
	\label{fig:4-star}
\end{figure}

\subsubsection{Graphs in $\mathcal C$}
\begin{theorem}
	In $\mathcal C^4$, only the graph of Figure~\ref{fig:4-star} can admit a \QHD smoothing.
\end{theorem}
\begin{proof}
	We have five $\Gamma_i$ exceptional cusps and the $C_i^j$ on the fourth arm.
	If there is no $C_\ell^1$, we are done; otherwise, there are $1+\ell$ many $C_n^\alpha$ curves.
	Counting points, we see that $C_\ell^1$ has three points (besides the double points of the $C_n^\alpha$): its double point $p_d$, and two simple points $p_{s_i}$.
	$(\Gamma_i,C_\ell^1)=6$ can only happen, if every $\Gamma_i$ has its double point at $p_d$ and contains both $p_{s_i}$ by Lemma \ref{gammalemma}.
	
	A simple checking of cases shows that all of the $\Gamma_i$ have to go through a further shared point (denoted  $p_\Gamma$), and each must  have a final point not on any other $\Gamma_j$ (denoted $p_i$).
	Every $C_i^j$ curve which is not $C_\ell^1$ or $C_n^\alpha$ will have to contain either $p_\Gamma$ or the $p_i$ to intersect the $\Gamma_i$ 6 times (but not both).
	Note, that a "switch" lemma, similar to Lemma \ref{switchlemma}  applies to these two possibilities as well, if some curves contain, say $p_\Gamma$, and the next curve contains the $p_i$, then every subsequent curve has to contain the $p_i$ as well, since the curves $C_i^j$ have to contain the previous ones' points except for at most one.
	
	Every further curve $C_i^j$ with $i<n$ has to have its double point at $p_d$ or one of the $p_{s_i}$.
	By Lemmata \ref{2-3lemma}-\ref{switchlemma} we have to consider the cases where some curves after $C_\ell^1$ have double points at $p_d$ or $p_{s_1}$ and possibly some later curves have their double points at $p_{s_1}$ in the former, or $p_{s_2}$ in the latter case.

	We construct a new system of curves as follows: take the points of $X:=C_\ell^1\setminus\{A\}$, where $A$ denotes the double point of the first curve after $C_\ell^1$. Take every curve with index $(i,j)\neq(\ell,1)$ or $(n,\alpha)$. Our new curves will be $C_i^j\setminus X$ if the double point of this curve was at $A$, or $C_i^j\setminus\{X\cup \{A\}\}$ if its double point was not at $A$, with the modification that the $p_i$ are combined into a single point $P$, and finally with an added new curve $S_0=\{p_\Gamma,P\}$.
	These curves give a combinatorial solution to the smoothing problem of a linear graph, thus in particular it has at least as many points as curves. From this we deduce, that $k+1\leq \tilde k+3$, where $k$ denotes the number of curves in the original configuration besides $C_\ell^1,C_n^\alpha,\Gamma_i$, and $\tilde k$ denotes the number of points in the original configuration besides the points of $C_\ell^1,p_\Gamma,p_1,...,p_5$.

	Suppose there existed a \QHD smoothing. Then we would have 
	$$1+\ell+5+1+k=1+\ell+3+1+5+\tilde k\leftrightarrow k=\tilde k+3.$$
	Thus $\tilde k+4\leq\tilde k+3$, contradiction. 
\end{proof}
\subsubsection{Graphs in $\mathcal B$}
\begin{theorem}
	In $\mathcal B^4$, only the known case admits a \QHD smoothing.
\end{theorem}
\begin{proof}
	We have $L_1,S_1,S_2,\Gamma_1 $ besides the $C_i^j$.
	If every blowup happens next to the leaf of the fourth arm, we are in the known case. 
	Otherwise there are $1+\ell$ many $C_n^\alpha$ curves, and $C_\ell^1$ has three points (which are not double points of the $C_n^\alpha$ curves), denoted $p_d,p_{s_1},p_{s_2}$ as before.
	Besides these three points, $\Gamma_1$ has two further points $p_\Gamma,p_1$.
	$L_1$ contains two points. These can be either $p_d$ and a new point $p_L$, or $p_{s_1},p_{s_2}$.
	In this latter case, there can be no further $C_i^j$, since  either $X(C_i^j)=p_d$ which forces $(C_i^j,L_1)=1$, or 
	$X(C_i^j)=p_{s_1}$, which implies $((C_i^j),L_1)=3$.
	This means that this case cannot occur since we would have fewer curves than points.
	The $S_i$ have to go through $p_d$ to intersect $C_\ell^1$ three times. From here, either they both go through $p_{s_1}$, and each has one extra point ($p_2$ and $p_3$ respectively), or they split the simple points ($p_{s_1}\in S_1$ and $p_{s_2} \in S_2$) and meet at a further point $p_S$.
	
	In the latter case, no further $C_i^j$ can exist. This is because $X(C_i^j)=p_d$ would be forced, meaning $p_{s_1}\in C_i^j$ and $p_{s_2}\not\in C_i^j$; thus $(C_i^j,S_2)=2$, which is a contradiction. Thus this case cannot happen, since there are at least two curves on the long arm, which are not on the end by the construction of the $\mathcal B$ class.

	If both $S_i$ go through $p_{s_1}$ then by Lemma~\ref{switchlemma} the sequence of double points for the subsequent $C_i^j$ curves must follow one of the following patterns:
	\begin{enumerate}
		\item $p_{s_1}$ then $p_{s_2}$,
		\item $p_d$ then $p_{s_1}$,
		\item $p_d$ then $p_{s_2}$,
		\item entirely at $p_{s_2}$.
	\end{enumerate} In the first three patterns the "switch" may or may bot occur.
	
	
	For all of these cases we derive a combinatorial smoothing for a linear graph using the same method as previously. Let $A$ denote $X(C_i^j)$ where $C_i^j$ is the next curve after $C_\ell^1$. Let $X:=(C_\ell^1\cup L_1\cup S_1\cup S_2)\setminus\{A\}$. If $X(C)=A$, then consider $C\setminus X$, otherwise we take a curve to be $C\setminus(X\cup\{A\})$.
	Finally add a curve $\{p_\Gamma,p_1\}$. The validity of this configuration shows that $k+1\leq \tilde k+3$.
	On the other hand we can calculate 
	$$1+\ell+5+k=1+\ell+8+\tilde k\leftrightarrow k=\tilde k+3,$$
	which is a contradiction.
	
\end{proof}
\subsubsection{Graphs in $\mathcal A$}
\begin{theorem}
	In $\mathcal A^4$, only the known case admits a \QHD smoothing.
\end{theorem}
\begin{proof}
	We have $\Gamma_1,\Gamma_2, S_1$ and the cusps.
	The known case is when every blowup happens towards the leaf.
	Otherwise,  $1+\ell$ is the number of $C_n^\alpha$ curves.
	Besides their double points $C_\ell^1$ has its double point and two simple points $p_d,p_{s_1},p_{s_2}$.
	$\Gamma_1$ and $\Gamma_2$ have their double points at $p_d$ to intersect $C_\ell^1$ enough times. They intersect at one additional point $p_\Gamma$, and each has a unique simple point each $p_1,p_2$.
	The curve $S_1$ contains $p_d, p_{s_1}$ and a further point $p_S$.
	
	The proof is essentially the same as the previous two cases.
	The double points of subsequent curves are restricted to at most two of the three possible points $p_d,p_{s_1},p_{s_2}$. Let $A$ denote the double point of the first curve.
	We create a solution to the combinatorial smoothing problem of a linear graph by identifying the points $p_1,p_2$ into a single point $P$, taking $X=C_\ell^1\setminus\{A\}$ and considering the curves $C_i^j\setminus X$ if $X(C_i^j)=A$ and $C_i^j\setminus(X\cup\{A\})$ if $X(C_i^j)\neq A$, and finally adding $S_0=\{p_\Gamma,P\}$ to the beginning of the configuration.
	This means that $k+1\leq \tilde k+3$.
	Assuming the existence of a \QHD smoothing, we compute
	$$1+\ell+4+k=1+\ell+7+\tilde k\Leftrightarrow k=\tilde k+3.$$
\end{proof}

\begin{remark}
	The known families already ruled out in e.g. {\cite[Proposition 4.2]{gay2012symplectic}} can also be ruled out from having a \QHD smoothing using this method.
	Similarly to Section \ref{sect:examples} one can find many graphs with a fixed sandwiched presentation admitting a combinatorial \QHD smoothing.
	Unlike the minimal rational case, this is even possible inside the $\ABC$ families.
	
	More generally, there are graphs which are rational, and have more than one $-2$ node in the $\ABC$ families, which don't have a sandwich presentation using only smooth and cusp curves, and in fact there are rational graphs in $\ABC$, which are not sandwiched at all.
\end{remark}

\nocite{*}
\bibliographystyle{amsplain}
\bibliography{references}
\end{document}